\documentclass{mfat}
\pagespan{361}{369}

\usepackage{mmap}
\usepackage[utf8]{inputenc}
\usepackage{graphicx} 

\theoremstyle{plain}
\newtheorem{thm}{Theorem}[section]
\newtheorem*{thm*}{Theorem}
\newtheorem{prop}{Proposition}[section]
\newtheorem*{prop*}{Proposition}
\newtheorem{cor}{Corollary}[section]
\newtheorem*{cor*}{Corollary}

\newtheorem*{lem*}{Lemma}
\theoremstyle{definition}

\newtheorem*{defn*}{Definition}

\newtheorem*{exmp*}{Example}

\newtheorem*{exmps*}{Examples}

\newtheorem*{rem*}{Remark}

\newtheorem*{rems*}{Remarks}

\newtheorem*{note*}{Note}
\newcommand{\N}{{\mathbb N}}%%%%%%%%%NATURAL NUMBERS
\newcommand{\Z}{{\mathbb Z}}%%%%%%%%%INTEGERS
\newcommand{\R}{{\mathbb R}}%%%%%%%%%REAL NUMBERS
\newcommand{\C}{{\mathbb C}}%%%%%%%%%COMPLEX NUMBERS

\renewcommand{\theequation}{\thesection.\arabic{equation}}%%%%%%%%%EQUATION NUMBERING WITHIN SECTIONS

\begin{document}

\title[On the Carleman ultradifferentiable vectors]
    {On the Carleman ultradifferentiable vectors  of a scalar type
spectral operator}

   % \title[On the Carleman ultradifferentiable vectors]
%{On the Carleman ultradifferentiable vectors\\ of a scalar type
%spectral operator}

\author{Marat V. Markin}
\address{Department of Mathematics, California State University, Fresno 5245 N. Backer Avenue,
M/S PB 108 Fresno, CA 93740-8001}
%\curraddr{}
\email{mmarkin@csufresno.edu}
%\thanks{Thank you.}

%    General info
\date{04/06/2015}
\dedicatory{To Academician Yu. M. Berezansky in honor of his 90th jubilee}

\begin{abstract}
  A description of the Carleman classes of vectors, in particular the
  Gevrey classes, of a scalar type spectral operator in a reflexive
  complex Banach space is shown to remain true without the reflexivity
  requirement. A similar nature description of the entire vectors of
  exponential type, known for a normal operator in a complex Hilbert
  space, is generalized to the case of a scalar type spectral operator
  in a complex Banach space.
\end{abstract}

\maketitle

%%%%%%%%%%%%%%%%%%%%%%%%%EPIGRATH%%%%%%%%%%%%%%%%%%%%%%%%%%%%
%\epigraph
\begin{flushright}
	\rm{Never cut what you can untie}.\\ \it{Joseph\ Joubert}
\end{flushright}

\medskip

\section[Introduction]{Introduction}

The description of the \textit{Carleman classes} of ultradifferentiable vectors, in particular the \textit{Gevrey classes}, of a \textit{normal operator} in a complex Hilbert space in terms of its {\it spectral measure} established in \cite{GorV83} (see also \cite{Gor-Knyaz} and \cite{book}) is generalized in {\cite[Theorem 3.1]{Markin2004(2)}} to the case of a {\it scalar type spectral operator} in a complex {\it reflexive} Banach space.

Here, the reflexivity requirement is shown to be superfluous and a similar nature description of the entire vectors of exponential type, known for a normal operator in a complex Hilbert space (see, e.g., \cite{Gor-Knyaz}), is generalized to the case of a scalar type spectral operator in a complex Banach space.

%%%%%%%%%%%%%%%%%%%%%%%%%%%%%%%%%%%%%%%%%%%%%%%%%%%%%%%%%%%%%
\section[Preliminaries]{Preliminaries}

For the reader's convenience, we shall outline in this section certain essential preliminaries.

%%%%%%%%%%%%%%%%%%%%%%%%%%%%%%%%%%%%%%%%%%%%%%%%%%%%%%%%%%%%%
\subsection{Scalar type spectral operators}

Henceforth, unless specified otherwise, $A$ is supposed to be a {\it scalar type spectral operator} in a complex Banach space $(X,\|\cdot\|)$ and $E_A(\cdot)$ to be its {\it spectral measure} (the {\it resolution of the identity}), the operator's \textit{spectrum} $\sigma(A)$ being the {\it support} for the latter \cite{Survey58,Dun-SchIII}.

In a complex Hilbert space, the scalar type spectral operators are precisely those similar to the {\it normal} ones \cite{Wermer}.

A scalar type spectral operator in complex Banach space has an {\it operational calculus} analogous to that of a \textit{normal operator} in a complex Hilbert space \cite{Survey58,Dun-SchII,Dun-SchIII}. To any Borel measurable function $F:\C\to \C$ (or $F:\sigma(A)\to \C$, $\C$ is the \textit{complex plane}), there corresponds a scalar type spectral operator
\begin{equation*}
F(A):=\int_\C F(\lambda)\,dE_A(\lambda)
=\int_{\sigma(A)} F(\lambda)\,dE_A(\lambda)
\end{equation*}
defined as follows:
%\begin{equation*}
$$
\begin{aligned}
F(A)f&:=\lim_{n\to\infty}F_n(A)f,\quad  f\in D(F(A)),\\
D(F(A))&:=\left\{f\in X
\big| \lim_{n\to\infty}F_n(A)f\ \text{exists}\right\}
\end{aligned}
$$
%\end{equation*}
($D(\cdot)$ is the {\it domain} of an operator), where
\begin{equation*}
F_n(\cdot):=F(\cdot)\chi_{\{\lambda\in\sigma(A)\,|\,|F(\lambda)|\le n\}}(\cdot),
\quad n\in\N,
\end{equation*}
($\chi_\delta(\cdot)$ is the {\it characteristic function} of a set $\delta\subseteq \C$, $\N:=\left\{1,2,3,\dots\right\}$ is the set of \textit{natural numbers}) and
\begin{equation*}
F_n(A):=\int_{\sigma(A)} F_n(\lambda)\,dE_A(\lambda),\quad n\in\N,
\end{equation*}
are {\it bounded} scalar type spectral operators on $X$ defined in the same manner as for a \textit{normal operator} (see, e.g., \cite{Dun-SchII,Plesner}).

In particular,
\begin{equation}\label{A}
A^n=\int_{\C} \lambda^n\,dE_A(\lambda)
=\int_{\sigma(A)} \lambda^n\,dE_A(\lambda),\quad n\in\Z_+,
\end{equation}
($\Z_+:=\left\{0,1,2,\dots\right\}$ is the set of \textit{nonnegative integers}).

The properties of the {\it spectral measure}  $E_A(\cdot)$
and the {\it operational calculus}, exhaustively delineated in \cite{Survey58,Dun-SchIII}, underly the entire subsequent discourse. Here, we shall outline a few facts of particular importance.

Due to its {\it strong countable additivity}, the spectral measure $E_A(\cdot)$ is {\it bounded} \cite{Dun-SchI,Dun-SchIII}, i.e., there is such an $M>0$ that, for any Borel set $\delta\subseteq \C$,
\begin{equation}\label{bounded}
\|E_A(\delta)\|\le M.
\end{equation}

The notation $\|\cdot\|$ has been recycled here to designate the norm in the space $L(X)$ of all bounded linear operators on $X$. We shall adhere to this rather common economy of symbols in what follows adopting the same notation for the norm in the \textit{dual space} $X^*$ as well.

%%%%%%%%%%%%%%%%%%%%%%%%%%%%%%%%%%%%%%%%%%%%%%%%%%

For any $f\in X$ and $g^*\in X^*$, the \textit{total variation} $v(f,g^*,\cdot)$ of the complex-valued Borel measure $\langle E_A(\cdot)f,g^* \rangle$
($\langle \cdot,\cdot \rangle$ is the {\it pairing} between the space $X$ and its dual $X^*$) is a {\it finite} positive Borel measure with
\begin{equation}\label{tv}
v(f,g^*,\C)=v(f,g^*,\sigma(A))\le 4M\|f\|\|g^*\|
\end{equation}
(see, e.g., \cite{Markin2004(2)}).
%%%%%%%%%%%%%%%%%%%%%%%%%%%%%%%%%%%%%%%%%%%%%%%%%%%%%%%%%
Also (Ibid.), $F:\C\to \C$ (or $F:\sigma(A)\to \C$) being an arbitrary Borel measurable function, for any $f\in D(F(A))$, $g^*\in X^*$, and an arbitrary Borel set $\sigma\subseteq \C$,
\begin{equation}\label{cond(ii)}
\int_\sigma|F(\lambda)|\,dv(f,g^*,\lambda)
\le 4M\|E_A(\sigma)F(A)f\|\|g^*\|.
\end{equation}
In particular,
\begin{equation}\label{cond(i)}
\int_{\C}|F(\lambda)|\,d v(f,g^*,\lambda)
=\int_{\sigma(A)}|F(\lambda)|\,d v(f,g^*,\lambda)\le 4M\|F(A)f\|\|g^*\|.
\end{equation}
The constant $M>0$ in \eqref{tv}--\eqref{cond(i)} is from
\eqref{bounded}.

The following statement allowing to characterize the domains of the Borel measurable functions of a scalar type spectral operator in terms of positive Borel measures is also fundamental for our discussion.

%%%%%%%%%%%%%%%%%%%%%%%%%%%%%%%%%%%
\begin{prop}%[
{\rm(}{\cite[Proposition $3.1$]{Markin2002(1)}}{\rm)}
\label{prop}\!.\
Let $A$ be a scalar type spectral operator in a complex Banach space $(X,\|\cdot\|)$ and $F:\C\to \C$
(or $F:\sigma(A)\to \C$) be Borel measurable function.
Then $f\in D(F(A))$ iff
\begin{enumerate}
\item[(i)] For any $g^*\in X^*$,
$\displaystyle \int_{\sigma(A)} |F(\lambda)|\,d v(f,g^*,\lambda)<\infty$.
%%%%%%%%%%%%%%%%%%%%%%%%%%%%%%%%%%%%%%%%%%%%%%%%%%%%%%%%%%%%%%
\item[(ii)] $\displaystyle \sup_{\{g^*\in X^*\,|\,\|g^*\|=1\}}
\int_{\{\lambda\in\sigma(A)\,|\,|F(\lambda)|>n\}}
|F(\lambda)|\,dv(f,g^*,\lambda)\to 0 \quad \text{as}\quad n\to\infty$.
\end{enumerate}
\end{prop}

%%%%%%%%%%%%%%%%%%%%%%%%%%%%%%%%%%%%%%%%%%%%%%%%%%%%%%%%%%
Subsequently, the frequent terms {\it "spectral measure"} and {\it "operational calculus"} will be abbreviated to {\it s.m.} and {\it o.c.}, respectively.

%%%%%%%%%%%%%%%%%%%%%%%%%%%%%%%%%%%%%%%%%%%%%%%%%%%%%%%%%%%%%
\subsection{The Carleman classes of vectors}\
Let $A$ be a densely defined closed linear operator in a complex Banach space $(X,\|\cdot\|)$ and $\left\{m_n\right\}_{n=0}^\infty$ be a sequence of positive numbers and
\begin{equation*}
C^{\infty}(A):=\bigcap_{n=0}^{\infty}D(A^n).
\end{equation*}
The subspaces of $C^{\infty}(A)$
%\begin{equation*}
$$
\begin{aligned}
C_{\{m_n\}}(A)&:=\left\{f\in C^{\infty}(A) \big |
\exists \alpha>0\ \exists c>0:
\|A^nf\| \le c\alpha^n m_n,\ n\in\Z_+ \right\},\\
C_{(m_n)}(A)&:=\left\{f \in C^{\infty}(A) \big | \forall \alpha > 0 \ \exists c>0:
\|A^nf\| \le c\alpha^n m_n,\ n\in\Z_+ \right\}
\end{aligned}
%\end{equation*}
$$
are called the {\it Carleman classes} of ultradifferentiable vectors of the operator $A$ corres\-ponding to the sequence $\left\{m_n\right\}_{n=0}^\infty$
of {\it Roumieu} and {\it Beurling type}, respectively.

The inclusions
\begin{equation}\label{incl1}
C_{(m_n)}(A)\subseteq C_{\{m_n\}}(A)\subseteq C^\infty(A)\subseteq X
\end{equation}
are obvious.

If two sequences of positive numbers $\bigl\{m_n \bigr\}_{n=0}^\infty$ and $\bigl\{m'_n \bigr\}_{n=0}^\infty$ are related  as follows:
\begin{equation*}
\forall \gamma > 0 \ \exists c=c(\gamma)>0:\
m'_n \le c\gamma^n m_n,\quad  n\in\Z_+,
\end{equation*}
we also have the inclusion
\begin{equation}\label{incl2}
C_{\{m'_n\}}(A) \subseteq C_{(m_n)}(A),
\end{equation}
the sequences being subject to the condition
\begin{equation*}
\exists \gamma_1,\gamma_2 > 0, \ \exists c_1,c_2>0:\
c_1\gamma_1^n m_n\le m'_n \le c_2\gamma_2^n m_n,\quad n\in\Z_+,
\end{equation*}
their corresponding Carleman classes coincide
\begin{equation}\label{equal}
C_{\{m_n\}}(A)=C_{\{m'_n\}}(A),\quad  C_{(m_n)}(A)=C_{(m'_n)}(A).
\end{equation}

Considering {\it Stirling's formula} and the latter,
\begin{equation*}
\begin{aligned}
{\mathcal E}^{\{\beta\}}(A) & :=C_{\{[n!]^\beta\}}(A)
=C_{\{n^{\beta n}\}}(A),\\
{\mathcal E}^{(\beta)}(A) & :=C_{([n!]^\beta)}(A)
=C_{(n^{\beta n})}(A)
\end{aligned}
\end{equation*}
with $\beta\ge 0$ are the well-known \textit{Gevrey classes} of strongly ultradifferentiable vectors of $A$ of order $\beta$ of Roumieu and Beurling type, respectively (see, e.g., \cite{GorV83,book,Gor-Knyaz}). In particular,
${\mathcal E}^{\{1\}}(A)$ and ${\mathcal E}^{(1)}(A)$ are the well-known classes of {\it analytic} and {\it entire} vectors of $A$, respectively \cite{Goodman,Nelson}; ${\mathcal E}^{\{0\}}(A)$ and ${\mathcal E}^{(0)}(A)$ (i.e., the classes $C_{\{1\}}(A)$ and $C_{(1)}(A)$ corresponding to the sequence $m_n\equiv 1$) are the classes of \textit{entire} vectors of \textit{exponential} and \textit{minimal exponential type}, respectively (see, e.g., \cite{Radyno1983(1),Gor-Knyaz}).

If the sequence of positive numbers $\left\{m_n\right\}_{n=0}^\infty$ satisfies the condition
\begin{equation}\label{WGR}
\textbf{(WGR)}\ \forall \alpha>0\ \exists c=c(\alpha)>0:\
c\alpha^n \le m_n,\quad  n\in\Z_+,
\end{equation}
the scalar function
\begin{equation}\label{T}
T(\lambda):=m_0\sum_{n=0}^{\infty} \frac{\lambda^n}{m_n},\quad  \lambda\ge 0,\quad (0^0:=1)
\end{equation}
first introduced by S. Mandelbrojt \cite{Mandel}, is well-defined (cf. \cite{Gor-Knyaz}). The function is {\it continuous}, {\it strictly increasing}, and $T(0)=1$.

As is shown in \cite{GorV83} (see also \cite{Gor-Knyaz} and \cite{book}), the sequence $\left\{m_n\right\}_{n=0}^\infty$ satisfying the condition \textbf{(WGR)}, for a {\it normal operator} $A$ in a complex Hilbert space $X$, the equalities
\begin{equation}\label{CCeq}
\begin{aligned}
C_{\{m_n\}}(A) & =\bigcup_{t>0}D(T(t|A|)),\\
C_{(m_n)}(A) & =\bigcap_{t>0}D(T(t|A|))\\
\end{aligned}
\end{equation}
are true, the normal operators $T(t|A|)$, $t>0$, defined in the sense of the operational calculus for a normal operator (see, e.g., \cite{Dun-SchII,Plesner}) and the function $T(\cdot)$ being replaceable with any {\it nonnegative}, {\it continuous}, and {\it increasing} on $[0,\infty)$ function $F(\cdot)$ satisfying
\begin{equation}\label{replace}
c_1F(\gamma_1\lambda)\le T(\lambda)\le c_2F(\gamma_2\lambda),\quad \lambda\ge R,
\end{equation}
with some $\gamma_1,\gamma_2,c_1,c_2>0$ and $R\ge 0$, in particular, with
\begin{equation*}
S(\lambda):=m_0\sup_{n\ge 0}\dfrac{\lambda^n}{m_n},
\quad  \lambda\ge 0,
\quad \text{or}\quad
P(\lambda):=m_0\biggl[\sum_{n=0}^\infty\dfrac{\lambda^{2n}}{m_n^2}\biggr]^{1/2},
\quad  \lambda\ge 0,
\end{equation*}
(cf. \cite{Gor-Knyaz}).

In {\cite[Theorem 3.1]{Markin2004(2)}}, the above is generalized to the case of a \textit{scalar type spectral operator} $A$ in a \textit{reflexive} complex Banach space $X$. The reflexivity requirement dropped, proved were the inclusions
\begin{equation}\label{CCincl}
\begin{aligned}
C_{\{m_n\}}(A) & \supseteq \bigcup_{t>0}D(T(t|A|)),\\
C_{(m_n)}(A) & \supseteq \bigcap_{t>0}D(T(t|A|))\\
\end{aligned}
\end{equation}
only, which is a deficiency for statements like {\cite[Theorem $5.1$]{Markin2008}} and {\cite[Theorem $3.2$]{Markin2015(2)}}.

%%%%%%%%%%%%%%%%%%%%%%%%%%%%%%%%%%%%%%%%%%%%%%%%%%%%%%%%%%%%%
\section{The Carleman classes of a scalar type spectral operator}

\begin{thm}\label{thm}\
Let $\bigl\{m_n\bigr\}_{n=0}^\infty$ be a sequence of positive numbers satisfying the condition {\bf (WGR)} (see \eqref{WGR}). Then, for a scalar type spectral operator $A$ in a complex Banach space $(X,\|\cdot\|)$, equalities \eqref{CCeq} are true, the scalar type spectral operators $T(t|A|)$, $t>0$, defined in the sense of the operational calculus for a scalar type spectral operator and the function $T(\cdot)$ being replaceable with any {\it nonnegative}, {\it continuous}, and {\it increasing} on $[0,\infty)$ function $F(\cdot)$ satisfying \eqref{replace}.
\end{thm}

\begin{proof}
We are only to prove the inclusions inverse to
\eqref{CCincl}, the rest, including the latter, having been proved in {\cite[Theorem 3.1]{Markin2004(2)}}.

Consider an arbitrary vector $f\in C_{\{m_n\}}(A)$ ($f\in C_{(m_n)}(A)$). Then necessarily, $f\in C^\infty(A)$  and
for a certain $\alpha>0$ (an arbitrary $\alpha>0$), there is a $c>0$ such that
\begin{equation}\label{CE}
\|A^nf\|\le c\alpha^n m_n,\quad n\in\Z_+,
\end{equation}
(see Preliminaries).

For any $g^*\in X^*$,
\begin{multline}\label{(i)}
\int_{\sigma(A)}
T\left(\tfrac{1}{2\alpha}|\lambda|\right)\,d v(f,g^*,\lambda)
=\int_{\sigma(A)} \sum_{n=0}^\infty\dfrac{|\lambda|^n}{2^n\alpha^n m_n}
\,d v(f,g^*,\lambda)
\\
\hfill
\text{by the {\it Monotone Convergence Theorem};}
\\
\shoveleft{
=\sum_{n=0}^\infty\int_{\sigma(A)}\dfrac{|\lambda|^n}{2^n\alpha^n m_n}
\,d v(f,g^*,\lambda)
=\sum_{n=0}^\infty\dfrac{1}{2^n\alpha^n m_n}
\int_{\sigma(A)}|\lambda|^n\,d v(f,g^*,\lambda)
}\\
\hfill
\text{by \eqref{cond(i)} and \eqref{A};}
\end{multline}
\begin{multline*}
\qquad\;\;\le \sum_{n=0}^\infty\dfrac{1}{2^n\alpha^n m_n}4M\|A^n f\|\|g^*\|
\\
\hfill
\text{by \eqref{CE};}
\\
\le 4Mc\sum_{n=0}^\infty \dfrac{1}{2^n}\|g^*\|=8Mc\|g^*\|<\infty.
\end{multline*}

For an arbitrary $\varepsilon>0$, one can fix an $N\in\N$ such that,
\begin{equation}\label{small1}
\dfrac{M^2c}{2^{N-2}}<\varepsilon/2.
\end{equation}

Due to the strong continuity of the {\it s.m.}, for any
$n\in\N$,
\begin{equation*}
\left\|E_A\left({\left\{\lambda\in\sigma(A)\big | T(\tfrac{1}{2\alpha}|\lambda|)>k\right\}}\right)A^n f\right\|
\to 0\quad \text{as}\quad  k\to\infty.
\end{equation*}
Hence, there is a $K\in\N$ such that
\begin{equation}\label{small2}
\sum_{n=0}^N\dfrac{1}{2^n\alpha^n m_n}4M\left\|E_A\left({\left\{\lambda\in\sigma(A)\big | T(\tfrac{1}{2\alpha}|\lambda|)>k\right\}}\right)A^n f\right\|
<\varepsilon/2
\end{equation}
whenever $k\ge K$.

Similarly to \eqref{(i)}, for $k\ge K$, we have
\begin{multline*}
\sup_{\{g^*\in X^*\,|\,\|g^*\|=1\}}\int_{\left\{\lambda\in\sigma(A)\big | T(\frac{1}{2\alpha}|\lambda|)>k\right\}}
T\left(\tfrac{1}{2\alpha}|\lambda|\right)\,d v(f,g^*,\lambda)
\\
\shoveleft{
=\sup_{\{g^*\in X^*\,|\,\|g^*\|=1\}}
\sum_{n=0}^\infty\dfrac{1}{2^n\alpha^n m_n}
\int_{\left\{\lambda\in\sigma(A)\big | T(\frac{1}{2\alpha}|\lambda|)>k\right\}}|\lambda|^n\,d v(f,g^*,\lambda)
}\\
\hfill
\text{by \eqref{cond(ii)} and \eqref{A};}
\\
\shoveleft{
\le \sup_{\{g^*\in X^*\,|\,\|g^*\|=1\}}
\sum_{n=0}^\infty\dfrac{1}{2^n\alpha^n m_n}4M\left\|E_A\left({\left\{\lambda\in\sigma(A)\big | T(\tfrac{1}{2\alpha}|\lambda|)>k\right\}}\right)A^n f\right\|\|g^*\|
}\\
\shoveleft{
\le \sup_{\{g^*\in X^*\,|\,\|g^*\|=1\}}\biggl[
\sum_{n=0}^N\dfrac{1}{2^n\alpha^n m_n}4M\left\|E_A\left({\left\{\lambda\in\sigma(A)\big | T(\tfrac{1}{2\alpha}|\lambda|)>k\right\}}\right)A^n f\right\|\|g^*\|
}\\
\shoveleft{
+\sum_{n=N+1}^\infty\dfrac{1}{2^n\alpha^n m_n}4M\left\|E_A\left({\left\{\lambda\in\sigma(A)\big | T(\tfrac{1}{2\alpha}|\lambda|)>k\right\}}\right)\right\|\left\|A^n f\right\|\|g^*\|\biggr]
}\\
\hfill
\text{by \eqref{bounded} and \eqref{CE};}
\\
\shoveleft{
\le \sup_{\{g^*\in X^*\,|\,\|g^*\|=1\}}\biggl[
\sum_{n=0}^N\dfrac{1}{2^n\alpha^n m_n}4M\left\|E_A\left({\left\{\lambda\in\sigma(A)\big |
T(\tfrac{1}{2\alpha}|\lambda|)>k\right\}}\right)A^n f\right\|\|g^*\|
}\\
+4M^2c\sum_{n=N+1}^\infty \dfrac{1}{2^n}\|g^*\|\biggr]
\\
\shoveleft{
\le \sum_{n=0}^N\dfrac{1}{2^n\alpha^n m_n}4M\left\|E_A\left({\left\{\lambda\in\sigma(A)\big |
T(\tfrac{1}{2\alpha}|\lambda|)>k\right\}}\right)A^n f\right\|
+\dfrac{M^2c}{2^{N-2}}
}\\
\hfill
\text{by \eqref{small1} and \eqref{small2};}
\\
\\
<\varepsilon/2+\varepsilon/2=\varepsilon
\end{multline*}
and we conclude that
\begin{equation}\label{(ii)}
\sup_{\{g^*\in X^*\,|\,\|g^*\|=1\}}\int_{\left\{\lambda\in\sigma(A)\big | T(\frac{1}{2\alpha}|\lambda|)>k\right\}}
T\left(\tfrac{1}{2\alpha}|\lambda|\right)\,d v(f,g^*,\lambda)
\to 0\quad \text{as}\quad k\to\infty.
\end{equation}

By Proposition \ref{prop}, \eqref{(i)} and \eqref{(ii)} imply
\begin{equation*}
f\in D(T(\tfrac{1}{2\alpha}|A|)).
\end{equation*}

Considering that for $f\in C_{\{m_n\}}(A)$, $\alpha>0$
is fixed and for $f\in C_{(m_n)}(A)$, $\alpha>0$ is arbitrary,
we infer that
\begin{equation*}
f\in \bigcup_{t>0}D(T(t|A|))
\end{equation*}
in the former case and
\begin{equation*}
f\in \bigcap_{t>0}D(T(t|A|))
\end{equation*}
in the latter.

Since $f\in C_{\{m_n\}}(A)$ ($f\in C_{(m_n)}(A)$) is arbitrary, we have proved the inclusions
\begin{equation*}
\begin{aligned}
C_{\{m_n\}}(A) & \subseteq \bigcup_{t>0}D(T(t|A|)),\\
C_{(m_n)}(A) & \subseteq \bigcap_{t>0}D(T(t|A|)),\\
\end{aligned}
\end{equation*}
which along with their inverses \eqref{CCincl} imply
equalities \eqref{CCeq} to be true.
\end{proof}

%%%%%%%%%%%%%%%%%%%%%%%%%%%%%%%%%%%%%%%%%%%%%%%%%%%%%%%%%%%%%
\section{The Gevrey classes}

The sequence $m_n:=[n!]^\beta$ ($m_n:=n^{\beta n}$) with $\beta>0$ satisfying the condition {\bf (WGR)} and the corresponding function $T(\cdot)$ being replaceable with $F(\lambda)=e^{\lambda^{1/\beta}}$, $\lambda\ge 0$, (see \cite{Markin2004(2)} for details, cf. also \cite{Gor-Knyaz}),
in {\cite[Corollary 4.1]{Markin2004(2)}} describing the Gevrey classes of vectors of a scalar type spectral operator in a \textit{reflexive} complex Banach space, the reflexivity requirement can be dropped as well and we have the following

\begin{cor}\label{cor}
Let $\beta>0$. Then, for a $A$ scalar type spectral operator in a complex Banach space $(X,\|\cdot\|)$,
\begin{equation*}
\begin{aligned}
{\mathcal E}^{\{\beta\}}(A) & =\bigcup_{t>0} D(e^{t|A|^{1/\beta}}),\\
{\mathcal E}^{(\beta)}(A) & =\bigcap_{t>0} D(e^{t|A|^{1/\beta}}).
\end{aligned}
\end{equation*}
\end{cor}

Corollary \ref{cor} generalizes the corresponding result of \cite{GorV83} (see also \cite{Gor-Knyaz,book}) for a {\it normal operator} $A$ in a complex Hilbert space and, for $\beta=1$, gives a description of the {\it analytic} and {\it entire} vectors of a scalar type spectral operator $A$ in a complex Banach space.

%%%%%%%%%%%%%%%%%%%%%%%%%%%%%%%%%%%%%%%%%%%%%%%%%%%%%%%%%%%%%
\section{The entire vectors of exponential type}

Observe that the sequence $m_n\equiv 1$ generating the entire vectors of exponential type does not meet the condition
\textbf{(WGR)} (see \eqref{WGR}) and thus, this case falls outside the realm of Theorem \ref{thm}.

As is known (cf., e.g., \cite{Gor-Knyaz,GorV-GorM1995}), for a \textit{normal operator} $A$ in a complex Hilbert space $X$,
\begin{equation*}
{\mathcal E}^{\{0\}}(A) = \bigcup_{\alpha>0}E_A(\Delta_\alpha)X
\end{equation*}
and
\begin{equation*}
{\mathcal E}^{(0)}(A)=\bigcap_{\alpha>0}E_A(\Delta_\alpha)X=E_A(\left\{0\right\})X=\ker A:=\left\{f\in X \big| Af=0\right\}
\end{equation*}
($E_A(\cdot)$ is the \textit{spectral measure} of $A$) with
\begin{equation*}
\Delta_\alpha:=\left\{\lambda\in\C\big | |\lambda|\le \alpha \right\},\quad \alpha>0.
\end{equation*}

We are to generalize the above to the case of a \textit{scalar type spectral operator} $A$ in a complex Banach space $X$.

\begin{thm}\label{exp}
For a scalar type spectral operator $A$ in a complex Banach space $(X,\|\cdot\|)$,
\begin{enumerate}
\item[(i)] $ {\mathcal E}^{\{0\}}(A) = \bigcup_{\alpha>0}E_A(\Delta_\alpha)X$,
%%%%%%%%%%%%%%%%%%%%%%%%%%%%%%%%%%%%%%%%%%%%%%%%%%%%%%%%%%%%%%
\item[(ii)] $
{\mathcal E}^{(0)}(A)=\bigcap_{\alpha>0}E_A(\Delta_\alpha)X=E_A(\left\{0\right\})X=\ker A:=\left\{f\in X \big | Af=0\right\}
$,
\end{enumerate}
where $\Delta_\alpha:=\left\{\lambda\in\C\big | |\lambda|\le \alpha \right\}$, $\alpha>0$.
\end{thm}

\begin{proof}
Let $f\in \bigcup_{\alpha>0}E_A(\Delta_\alpha)X$ ($f\in \bigcap_{\alpha>0}E_A(\Delta_\alpha)X$), i.e.,
$$f=E_A\left(\bigl\{\lambda\in\C \bigm| |\lambda|\le \alpha\bigr\}\right)f$$ for some (any) $\alpha>0$. Then, by the properties of the \textit{o.c.},
\begin{equation*}
f\in C^\infty(A).
\end{equation*}

Furthermore, as follows form the {\it Hahn-Banach Theorem} and the properties of the \textit{o.c.} (in particular, \eqref{A}),
\begin{multline*}
\qquad\qquad\|A^nf\|=\biggl\|\int_{\C}\lambda^n\,dE_A(\lambda)f \biggr\|
=\biggl\|\int_{\{\lambda\in\C | |\lambda|\le \alpha\}}\lambda^n\,dE_A(\lambda)f \biggr\|
\\
\qquad\qquad\qquad\quad\shoveleft{
=\sup_{g^*\in X^*,\,\|g^*\|=1}\biggl|\biggl\langle \int_{\{\lambda\in\C | |\lambda|\le \alpha\}}\lambda^n
\,dE_A(\lambda)f,g^* \biggr\rangle\biggr| }
\\
\qquad\qquad\qquad\quad\shoveleft{
=\sup_{g^*\in X^*,\,\|g^*\|=1}\biggl|\int_{\{\lambda\in\C | |\lambda|\le \alpha\}}\lambda^n\,d\langle E_A(\lambda)f,g^*\rangle\biggr|
}\\
\qquad\qquad\qquad\quad\shoveleft{
\le \sup_{g^*\in X^*,\,\|g^*\|=1}\int_{\{\lambda\in\C | |\lambda|\le \alpha\}}|\lambda|^n\,dv(f,g^*,\lambda)
}\\
\qquad\qquad\qquad\quad\shoveleft{
\le \sup_{g^*\in X^*,\,\|g^*\|=1}\alpha^nv(f,g^*,\{\lambda\in\C | |\lambda|\le \alpha\})
}\\
\hfill \text{by \eqref{tv};}
\\
\qquad\qquad\qquad\,\shoveleft{
\le \sup_{g^*\in X^*,\,\|g^*\|=1}4M\|f\|\|g^*\|\alpha^n \le 4M\left[\|f\|+1\right]\alpha^n,\quad \alpha>0,
}
\end{multline*} 
which implies that $f\in {\mathcal E}^{\{0\}}(A)$
($f\in {\mathcal E}^{(0)}(A)$).

Conversely, for an arbitrary $f\in {\mathcal E}^{\{0\}}(A)$ ($f\in {\mathcal E}^{(0)}(A)$),
\begin{equation*}
\|A^nf\|\le c\alpha^n,\quad n\in\Z_+,
\end{equation*}
with some (any) $\alpha>0$ and some $c>0$.

Then, for any $\gamma>\alpha$ and $g^*\in X^*$, we have
\begin{multline*}
\gamma^n v(f,g^*,\{\lambda\in\C||\lambda|\ge\gamma\})
\le \int_{\{\lambda\in\C||\lambda|\ge\gamma\}}|\lambda|^n\,dv(f,g^*,\lambda)
\\
\shoveleft{
\le \int_{\C}|\lambda|^n\,dv(f,g^*,\lambda)
\hfill \text{by \eqref{cond(i)};}
}\\
\le 4M\|A^nf\|\|g^*\|\le 4Mc\|g^*\|\alpha^n,\quad n\in\Z_+.
\end{multline*}

Therefore,
\begin{equation*}
v(f,g^*,\{\lambda\in\C||\lambda|\ge\gamma\})
\le  4Mc\|g^*\|\biggl(\dfrac{\alpha}{\gamma}\biggr)^n,\quad n\in\Z_+.
\end{equation*}

Considering that $\alpha/\gamma<1$ and passing to the limit as $n\to\infty$, we conclude that
\begin{equation*}
v(f,g^*,\{\lambda\in\C||\lambda|\ge\gamma\})=0,\quad g^*\in X^*,
\end{equation*}
and the more so
\begin{equation*}
\langle E_A(\{\lambda\in\C||\lambda|\ge\gamma\})f,g^*\rangle=0,\quad g^*\in X^*.
\end{equation*}

Whence, as follows from the {\it Hahn-Banach Theorem},
\begin{equation*}
E_A(\{\lambda\in\C||\lambda|\ge\gamma\})f=0,
\end{equation*}
which, considering that $\gamma>\alpha$ is arbitrary, by the \textit{strong continuity} of the {\it s.m.}, implies that
\begin{equation*}
E_A(\{\lambda\in\C||\lambda|>\alpha\})f=0.
\end{equation*}

Hence, by the \textit{additivity} of the {\it s.m.},
\begin{equation*}
f=E_A(\{\lambda\in\C||\lambda|\le \alpha\})f+E_A(\{\lambda\in\C||\lambda|>\alpha\})f
=E_A(\{\lambda\in\C||\lambda|\le \alpha\})f,
\end{equation*}
which implies that $ f\in \bigcup_{\alpha>0}E_A(\Delta_\alpha)X$ ($ f\in \bigcap_{\alpha>0}E_A(\Delta_\alpha)X$).
%%%%%%%%%%%%%%%%%%%%%%%%%%%%%%%%%%%%%%%%%%%%%%%%%%%%%%%%%%%%%%
\end{proof}

An immediate implication of Theorem \ref{exp} is the following generalization of the well-known result on the denseness of exponential type vectors of a normal operator in a complex Hilbert space (see, e.g., \cite{Gor-Knyaz}), which readily follows by the \textit{strong continuity} of the {\it s.m.} and joins a number of similar results of interest for approximation and qualitative theories (see \cite{Radyno1983(1),Gor-Knyaz,GorV-GorM1995,GorV-GorM2005,GorV-GorM2006}).

\begin{cor}
For a scalar-type spectral operator $A$ in a complex Banach space $(X,\|\cdot\|)$,
\begin{equation*}
\overline{{\mathcal E}^{\{0\}}(A)}=X
\end{equation*}
($\overline{\cdot}$ is the \textit{closure} of a set in the strong topology of $X$).
\end{cor}

Hence, for any positive sequence $\left\{m_n\right\}_{n=0}^\infty$ satisfying the condition {\bf (WGR)}, in particular, for $m_n=[n!]^{\beta}$
with $\beta>0$, due to inclusion \eqref{incl2}
with $m_n'\equiv 1$,
\begin{equation*}
{\mathcal E}^{\{0\}}(A)\subseteq C_{(m_n)}(A),
\end{equation*}
which implies
\begin{equation*}
\overline{C_{(m_n)}(A)}=X.
\end{equation*}

%%%%%%%%%%%%%%%%%%%%%%%%%%%%%%%%%%%%%%%%%%%%%%%%%%%%%%%%%%%%%
\section{Final remarks}

Observe that, for a normal operator in a complex Hilbert,
equalities \eqref{CCeq} have not only the set-theoretic but also a topological meaning \cite{GorV83} (see also \cite{book,Gor-Knyaz}). By analogy, this also appears to be true for a scalar type spectral operator in a complex Banach space, although the idea was entertained by the author neither in \cite{Markin2004(2)} nor here.

For a normal operator in a complex Hilbert space, Theorems \ref{exp} and \ref{thm} can be considered as generalizations of \textit{Paley-Wiener Theorems} relating the smoothness of a square-integrable on the real axis $\R$ function $f(\cdot)$ to the decay of its \textit{Fourier transform} $\hat{f}(\cdot)$ as $x\to \pm\infty$ \cite{Paley-Wiener}, which precisely corresponds to the case of the \textit{self-adjoint} differential operator $A=i\dfrac{d}{dx}$ ($i$ is the {\it imaginary unit}) in the complex Hilbert space $L_2(\R)$ \cite{Gor-Knyaz}. Observe that, in $L_p(\R)$ with $1\le p<\infty$, $p\neq 2$, the same operator fails to be spectral \cite{Farwig-Marschall} (the domain of $A=i\dfrac{d}{dx}$ in $X=L_p(\R)$, $1\le p<\infty$, is understood to be the subspace
$\displaystyle
W_p^1(\R):=\left\{f\in L_p(\R)\big |f(\cdot)\
\text{is \textit{absolutely continuous} on $\R$ and}\ f'\in L_p(\R) \right\}$).

Theorem \ref{thm} entirely substantiates the proof of the \textit{"only if" part} of {\cite[Theorem $5.1$]{Markin2008}}, where inclusions \eqref{CCincl} turn out to be insufficient, and of {\cite[Theorem $3.2$]{Markin2015(2)}}. It appears to be fundamental for qualitative results of this nature (cf. \cite{Markin2009}).

%%%%%%%%%%%%%%%%%%%%%%%%%%%%%%%%%%%%%%%%%%%%%%%%%%%%%%%%%%%%%

{\it{Acknowledgments}}. The author's utmost appreciation is to Drs. Miroslav L. and Valentina I. Gorbachuk, whose work and life have become a perpetual source of inspiration, ideas, and goodness for him.

%%%%%%%%%%%%%%%%Bibliography%%%%%%%%%%%%%%%%%%%%%%%%%%%%%%%%%

\end{document}

\bibitem{book}
{A. U. Thor},   %authors
\textit{Book},  %book
{bookinfo}, %bookinfo
{TBiMC},    %publ
{Kyiv},     %publaddr
{2004}.     %year

\bibitem{paper}
{A. U. Thor},   %authors
\textit{paper}, %paper
{Theory of Probability} %journal
\textbf{70} %vol
{(2004)},   %year
{no.~4},    %issue
{101--107}. %pages

\bibitem{inbook}
{A. U. Thor},   %authors
\textit{paper}, %paper
{inbook},   %book
{bookinfo}, %bookinfo
{vol.~120}, %vol
{2004},     %year
{pp.~206--215}. %pages

\fi

%%%%%%%%%%%%%%%%%%%%%%%%%%%%%%%%%%%%%%%%%%%%%%%%%%%%%%%%%%%%
\bibitem{Survey58}
{N. Dunford},   %authors
\textit{A survey of the theory of spectral operators}, %paper
{Bull. Amer. Math. Soc.}    %journal
\textbf{64} %vol
{(1958)},   %year
%{no.~4},   %issue
{217--274}. %pages
%%%%%%%%%%%%%%%%%%%%%%%%%%%%%%%%%%%%%%%%%%%%%%%%%%%%%%%%%%%%%
\bibitem{Dun-SchI}
{N. Dunford and J. T. Schwartz with the assistance of W. G. Bade and R. G. Bartle},    %authors
\textit{Linear Operators. {\rm{Part I\,:}} General Theory},  %book
%{Pure and Applied Mathematics, vol. 7},    %bookinfo
{Interscience Publishers},  %publ
{New York},     %publaddr
{1958}.     %year
%%%%%%%%%%%%%%%%%%%%%%%%%%%%%%%%%%%%%%%%%%%%%%%%%%%%%%%%%%%%%
\bibitem{Dun-SchII}
{\bysame},  %authors
\textit{Linear Operators. {\rm{Part II\,:}} Spectral Theory. Self Adjoint Operators in Hilbert Space}, %book
%{Ibid.},   %bookinfo
{Interscience Publishers},  %publ
{New York},     %publaddr
{1963}.     %year
%%%%%%%%%%%%%%%%%%%%%%%%%%%%%%%%%%%%%%%%%%%%%%%%%%%%%%%%%%%%%
\bibitem{Dun-SchIII}
{\bysame},  %authors
\textit{Linear Operators. {\rm{Part III\,:}} Spectral Operators}, %book
%{Ibid.},   %bookinfo
{Interscience Publishers},  %publ
{New York},     %publaddr
{1971}.     %year
%%%%%%%%%%%%%%%%%%%%%%%%%%%%%%%%%%%%%%%%%%%%%%%%%%%%%%%%%%%%%
\bibitem{Farwig-Marschall}
{R. Farwig and E. Marschall},   %authors
\textit{On the type of spectral operators and the nonspectrality of several differential operators on $L^p$},   %paper
{Integral Equations Operator Theory}    %journal
\textbf{4}  %vol
{(1981)},   %year
{no.~2},    %issue
{206--214}.   %pages
%%%%%%%%%%%%%%%%%%%%%%%%%%%%%%%%%%%%%%%%%%%%%%%%%%%%%%%%%%%%%
\bibitem{Goodman}
{R. Goodman},   %authors
\textit{Analytic and entire vectors for representations of Lie groups}, %paper
{Trans. Amer. Math. Soc.}   %journal
\textbf{143}    %vol
{(1969)},   %year
%{no.~4},   %issue
{55--76}.    %pages
%%%%%%%%%%%%%%%%%%%%%%%%%%%%%%%%%%%%%%%%%%%%%%%%%%%%%%%%%%%%%
\bibitem{GorV-GorM1995}
{M. L. Gorbachuk and V. I. Gorbachuk},    %authors
\textit{On the approximation of smooth vectors of a closed operator by entire vectors of exponential type},    %paper
{Ukrain. Mat. Zh.}    %journal
\textbf{47} %vol
{(1995)},   %year
{no.~5},    %issue
{616--628. (Ukrainian); English transl.} %pages
{Ukrainian Math.~J.}    %journal
\textbf{47} %vol
{(1995)},   %year
{no.~5},    %issue
{713--726}. %pages

%M.L. Gorbachuk and V.I. Gorbachuk, On the approximation of smooth vectors of a closed operator
%by entire vectors of exponential type. Ukrain. Mat. Zh. 47 (1995), no. 5, 616628. (Ukrainian) ..

%%%%%%%%%%%%%%%%%%%%%%%%%%%%%%%%%%%%%%%%%%%%%%%%%%%%%%%%%%%%%
\bibitem{GorV-GorM2005}
{\bysame},  %authors
\textit{On the well-posed solvability in some classes of entire functions of the Cauchy problem
for differential equations in a Banach space},  %paper
{Methods Funct. Anal. Topology}   %journal
\textbf{11} %vol
{(2005)},   %year
{no.~2},    %issue
{113--125}. %pages
%%%%%%%%%%%%%%%%%%%%%%%%%%%%%%%%%%%%%%%%%%%%%%%%%%%%%%%%%%%%%
\bibitem{GorV-GorM2006}
{\bysame},  %authors
\textit{On completeness of the set of root vectors for unbounded operators},    %paper
{Methods Funct. Anal. Topology}   %journal
\textbf{12} %vol
{(2006)},   %year
{no.~4},    %issue
{353--362}. %pages
%%%%%%%%%%%%%%%%%%%%%%%%%%%%%%%%%%%%%%%%%%%%%%%%%%%%%%%%%%%%%
\bibitem{GorV83}
{V. I. Gorbachuk},   %authors
\textit{Spaces of infinitely differentiable vectors of a nonnegative self-adjoint operator},    %paper
{Ukrain. Mat. Zh.}    %journal
\textbf{35} %vol
{(1983)},   %year
{no.~5},   %issue
{617--621. (Russian); English transl.} %pages
{Ukrainian Math.~J.}    %journal
\textbf{35} %vol
{(1983)},   %year
{no.~5},   %issue
{531--534}. %pages

%Gorbachuk V I, "Spaces of infinitely differentiable vectors of a non-negative self-adjoint operator",
%Ukrain. Mat. Zh., 35 1983, 617621; MR, 85i 47026; Ukrainian Math. J., 35 1983, 531534
%%%%%%%%%%%%%%%%%%%%%%%%%%%%%%%%%%%%%%%%%%%%%%%%%%%%%%%%%%%%%

\bibitem{book}
{V. I. Gorbachuk and M. L. Gorbachuk},    %authors
\textit{Boundary Value Problems for Operator Differential Equations},   %book
%{Mathematics and Its Applications (Soviet Series), vol.~48},    %bookinfo
{Kluwer Academic Publishers}, %publ
{Dordrecht---Boston---London},        %publaddr
{1991. (Russian edition: Naukova Dumka, Kiev, 1984)}     %year

\iffalse

V. I. Gorbachuk and M. L. Gorbachuk, Boundary Value Problems for Operator Differential.
Equations, Kluwer Academic Publishers, DordrechtBostonLondon, 1991

\bibitem{6}
{V. I. Gorbachuk and M. L. Gorbachuk},   %authors
\textit{Boundary Value Problems for Operator Differential Equations},  %book
{Kluwer Academic Publishers},    %publ
{Dordrecht---Boston---London},     %publaddr
{1991}. (Russian edition: Naukova Dumka, Kiev, 1984)    %year

\fi

%%%%%%%%%%%%%%%%%%%%%%%%%%%%%%%%%%%%%%%%%%%%%%%%%%%%%%%%%%%%%
\bibitem{Gor-Knyaz}
{V. I. Gorbachuk and A. V. Knyazyuk}, %authors
\textit{Boundary values of solutions of operator-differential equations}, %paper
{Russian Math. Surveys}   %journal
\textbf{44} %vol
{(1989)},   %year
{no.~3},   %issue
{67--111}.  %pages

%Gorbachuk, V.I. and Knyazyuk, A.V.: Boundary values of solutions of operator-
%differential equations, Russ. Math. Surveys 44, no. 3 (1989) 67-111.

%%%%%%%%%%%%%%%%%%%%%%%%%%%%%%%%%%%%%%%%%%%%%%%%%%%%%%%%%%%%%
\bibitem{Mandel}
{S. Mandelbrojt},   %authors
\textit{Series de Fourier et Classes Quasi-Analytiques de Fonctions},   %book
%{bookinfo},    %bookinfo
{Gauthier-Villars}, %publ
{Paris},        %publaddr
{1935}.     %year
%%%%%%%%%%%%%%%%%%%%%%%%%%%%%%%%%%%%%%%%%%%%%%%%%%%%%%%%%%%%%
\bibitem{Markin2002(1)}
{M. V. Markin},  %authors
\textit{On an abstract evolution equation with a spectral operator of scalar type}, %paper
{Int. J.~Math. Math. Sci.}  %journal
\textbf{32} %vol
{(2002)},   %year
{no.~9},    %issue
{555--563}. %pages

%%%%%%%%%%%%%%%%%%%%%%%%%%%%%%%%%%%%%%%%%%%%%%%%%%%%%%%%%%%%%

\bibitem{Markin2004(2)}
{\bysame},  %authors
\textit{On the Carleman classes of vectors of a scalar type spectral operator}, %paper
{Ibid.} %journal
\textbf{2004}   %vol
{(2004)},   %year
{no.~60},   %issue
{3219--3235}.   %pages
%%%%%%%%%%%%%%%%%%%%%%%%%%%%%%%%%%%%%%%%%%%%%%%%%%%%%%%%%%%%%
\bibitem{Markin2008}
{\bysame},  %authors
\textit{On scalar type spectral operators and Carleman ultradifferentiable $C_0$-se\-migroups}, %paper
{Ukrain. Mat. Zh.}    %journal
\textbf{60} %vol
{(2008)},   %year
{no.~9},    %issue
{1215--1233; English transl.}   %pages
{Ukrainian Math.~J.}    %journal
\textbf{60} %vol
{(2008)},   %year
{no.~9},    %issue
{1418--1436}.   %pages

%M. V. Markin, On scalar-type spectral operators and Carleman ultradifferentiable C 0-semigroups,
%Ukrainian Mathematical Journal, vol. 60, no. 9, pp. 14181436,

%%%%%%%%%%%%%%%%%%%%%%%%%%%%%%%%%%%%%%%%%%%%%%%%%%%%%%%%%%%%%
\bibitem{Markin2009}
{\bysame},  %authors
\textit{On the Carleman ultradifferentiability of weak solutions of an abstract evolution equation},    %paper
{Modern Analysis and Applications}, %book
{Oper. Theory Adv. Appl.},  %bookinfo
{vol.~191}, %vol
%{2009},        %year
{pp.~407--443}, %pages
{Birkh\"a\-user Verlag},    %publ
{Basel},        %publaddr
{2009}.
%%%%%%%%%%%%%%%%%%%%%%%%%%%%%%%%%%%%%%%%%%%%%%%%%%%%%%%%%%%%%
\bibitem{Markin2015(2)}
{\bysame},  %authors
\textit{On the generation of Beurling type Carleman ultradifferentiable $C_0$-semigroups by scalar type spectral operators},    %paper
{Methods Funct. Anal. Topology}   %journal
%\textbf{60}    %vol
%{(2015)},  %year
%{no.~9},   %issue
{(to appear)}.  %pages
%%%%%%%%%%%%%%%%%%%%%%%%%%%%%%%%%%%%%%%%%%%%%%%%%%%%%%%%%%%%%
\bibitem{Nelson}
{E. Nelson},    %authors
\textit{Analytic vectors},  %paper
{Ann. of Math. (2)} %journal
\textbf{70} %vol
{(1959)},   %year
{no.~3},    %issue
{572--615}. %pages
%%%%%%%%%%%%%%%%%%%%%%%%%%%%%%%%%%%%%%%%%%%%%%%%%%%%%%%%%%%%%%
\bibitem{Paley-Wiener}
{R. E. A. C. Paley and N. Wiener}, %authors
\textit{Fourier Transforms in the Complex Domain},  %book
{Amer. Math. Soc. Coll. Publ., vol.~19},    %bookinfo
{Amer. Math. Soc.}, %publ
{New York},     %publaddr
{1934}.     %year

%Paley, R.E.A.C., and Wiener, N., Fourier transforms in the complex domain, Amer. Math. Soc. Coll. Publ., Vol. 19, 1934

%%%%%%%%%%%%%%%%%%%%%%%%%%%%%%%%%%%%%%%%%%%%%%%%%%%%%%%%%%%%%

\bibitem{Plesner}
{A. I. Plesner}, %authors
\textit{Spectral Theory of Linear Operators},   %book
%{},    %bookinfo
{Nauka},    %publ
{Moscow},       %publaddr
{1965}.      %year
{(Russian)}
%%%%%%%%%%%%%%%%%%%%%%%%%%%%%%%%%%%%%%%%%%%%%%%%%%%%%%%%%%%%%
\bibitem{Radyno1983(1)}
{Ya. V. Radyno}, %authors
\textit{The space of vectors of exponential type},  %paper
{Dokl. Akad. Nauk BSSR} %journal
\textbf{27} %vol
{(1983)},   %year
{no.~9},    %issue
{791--793}.    %pages
{(Russian)}
%%%%%%%%%%%%%%%%%%%%%%%%%%%%%%%%%%%%%%%%%%%%%%%%%%%%%%%%%%%%%
\bibitem{Wermer}
{J. Wermer},    %authors
\textit{Commuting spectral measures on Hilbert space},  %paper
{Pacific J.~Math.}  %journal
\textbf{4}  %vol
{(1954)},   %year
{no.~3},    %issue
{355--361}. %pages
%%%%%%%%%%%%%%%%%%%%%%%%%%%%%%%%%%%%%%%%%%%%%%%%%%%%%%%%%%%%%
\end{thebibliography}
%%%%%%%%%%%%%%%%%%%%%%%%%
\end{document}